\documentclass{article}
\usepackage{amsmath}
\usepackage{amssymb}

\newcommand{\EE}{\mathbb{E}}

\newcommand{\NN}{\mathbb{N}}
\newcommand{\PP}{\mathbb{P}}
\newcommand{\RR}{\mathbb{R}}
\newcommand{\SaS}{\mathbb{S}}
\newcommand{\T}{\mathbb{T}}
\newcommand{\bP}{{\bf P}}

\newcommand{\D}{{\cal D}}

\newcommand{\hH}{{\cal H}}
\newcommand{\I}{{\cal I}}

\newcommand{\M}{{\cal M}}

\newcommand{\aS}{{\cal S}}

\newcommand{\htW}{{\widehat W}}

\newcommand{\bt}{\underline{\kappa}}
\newcommand{\ut}{\overline{\kappa}}
\newcommand{\Int}{{\rm int}}

\newtheorem{definition}{Definition}
\newtheorem{theorem}{Theorem}
\newtheorem{lemma}{Lemma}
\newtheorem{proof}{Proof}
\newtheorem{remark}{Remark}
\newtheorem{proposition}{Proposition}
\newtheorem{corollary}{Corollary}

\begin{document}
\begin{center}
{\Large STIT Tessellations have trivial tail $\sigma$-algebra}

\vspace{1cm}

{S. Mart{\'i}nez}\\
{Departamento Ingenier{\'\i}a Matem\'atica and Centro
Modelamiento Matem\'atico,\\  Universidad de Chile,\\
UMI 2807 CNRS, Casilla 170-3, Correo 3, Santiago, Chile.\\
Email: smartine@dim.uchile.cl} 

\vspace{0.5cm}

{Werner Nagel}\\
{Friedrich-Schiller-Universit\"at Jena,\\
Institut f\"ur Stochastik,\\
Ernst-Abbe-Platz 2,
D-07743 Jena, Germany.\\
Email: werner.nagel@uni-jena.de} 

\end{center}

\vspace{1cm}

\begin{abstract}
We consider homogeneous STIT tessellations $Y$ in the $\ell$-dimensional 
Euclidean space $\RR^\ell$ and show the triviality of the 
tail $\sigma$-algebra. This is a sharpening of the mixing result by 
Lachi{\`e}ze-Rey \cite{lr}.
\end{abstract}

Keywords: { Stochastic geometry; Random process of tessellations; 
Ergodic theory; tail $\sigma-$algebra}

AMS subject classification {60D05}{60J25, 60J75, 37A25}

\section{Introduction}
Let $Y=(Y_t: t > 0)$ be the STIT tessellation process, which is a Markov
process taking values in the space of tessellations of the $\ell$-dimensional 
Euclidean space $\RR^\ell$. The process $Y$ is spatially 
stationary (that is its law is invariant under translations of the space)
and on every polytope with nonempty interior $W$ (called a window) 
the induced tessellation process, denoted by 
$Y\wedge W=(Y_t\wedge W: t> 0)$, is a pure jump process. The process $Y$
was firstly constructed in \cite{nw} and in Section \ref{stitsection} we give 
a brief description of it and recall some of its main properties.

\medskip

In stochastic geometry, ergodic and mixing properties as well as
weak dependencies in space are studied. E.g., Heinrich et al. 
considered mixing properties 
for Voronoi and some other tessellations, Poisson cluster processes and 
germ-grain models and derived Laws of Large Numbers and Central Limit 
Theorems, see \cite{hein94, heinschm, heinmol}.

\medskip

For STIT tessellations, Lachi{\`e}ze-Rey \cite{lr} showed that they are 
mixing. Schreiber and Th{\"a}le \cite{schthaeLIM1, schthaeLIM2, schthaeLIM3} proved some 
limit theorems. They provided Central Limit Theorems for the number of vertices and the total edge length in the two-dimensional case ($\ell =2$). Furthermore, they proved that in dimensions $\ell \geq 3$ there appear non-normal limits, e.g. for the total surface area of the cells of STIT tessellations. 

\medskip

An issue is the problem of triviality of the tail $\sigma$-algebras for certain 
class of distributions. For random measures and point processes a key reference 
is the book by Daley and Vere-Jones (\cite{dvj}, pp. 207-209).

\medskip

In Section \ref{Sub1.1} of the present paper we introduce a definition for 
the tail $\sigma$-algebra ${\cal B}_{-\infty}(\T)$ on the space $\T$ of tessellations 
in the $\ell$-dimensional Euclidean space $\RR^\ell$. This definition relies 
essentially on the definition of the Borel $\sigma$-algebra with respect 
to the Fell topology on the set of closed subsets of $\RR^\ell$ (cf. \cite{sw}) 
as well as on the mentioned definition of the tail $\sigma$-algebra for 
random measures and point processes, as given in \cite{dvj}.

\medskip

Our main result is formulated in in Section \ref{trivitail}, Theorem \ref{trivtail}: 
For the distribution of a STIT tessellation, the 
tail $\sigma$-algebra ${\cal B}_{-\infty}(\T)$ is trivial, 
i.e. all its elements (the terminal events) have either probability 1 or 0. 
A detailed proof is given in Section \ref{proofs}.

\medskip

Finally, we compare STIT tessellations with Poisson hyperplane tessellations 
(PHT), and we show, that the tail $\sigma$-algebra ${\cal B}_{-\infty}(\T)$ 
is not trivial with respect to the distribution of PHT.

\section{The STIT model}\label{stitsection}

\subsection{A construction of STIT tessellations} \label{constr} Let $\RR^\ell$ 
be the $\ell$-dimensional Euclidean space, denote by $\T$ the space of 
tessellations of this space as defined in \cite{sw} (Ch. 10, Random Mosaics). A 
tessellation can be considered as a set $T$ of polytopes (the cells) with 
disjoint interiors and covering the Euclidean space, as well as a closed subset 
$\partial T$ which is the union of the cell boundaries. There is an obvious 
one-to-one relation between both ways of description of a tessellation, and their 
measurable structures can be related appropriately, see \cite{sw, martnag}.

\medskip

A compact convex polytope $W$ with non-empty interior in $\RR^\ell$,
is called a {\it window}. We can consider tessellations of $W$ and 
denote the set of all those tessellations by $\T\wedge W$. If
$T\in \T$ we denote by $T\land W$ the induced tessellation on $W$.
Its boundary is defined by 
$\partial (T\wedge W)= (\partial T \cap W)\cup {\partial W}$.

\medskip

In the present paper we will refer to the construction  given in 
\cite{nw} in all detail (for an alternative but equivalent construction see \cite{martnag}). On every window $W$ there exists
$Y\land W =(Y_t \land W: t > 0)$ a STIT tessellation process, 
it turns out to be a pure jump Markov process and hence has the strong Markov property 
(see \cite{brei}, Proposition 15.25). Each  
marginal $Y_t\land W$ takes values in $\T\wedge W$. 
Furthermore, for any $t>0$ the law of $Y_t$ is consistent with
respect to the windows, that is if $W'$ and $W$ are windows such that 
$W'\subseteq W$, then $(Y_t \wedge W) \wedge W' \sim Y_t \wedge W'$, where 
$\sim$ denotes the identity of distributions (for a proof see \cite{nw}). 
This yields the existence
of a STIT tessellation $Y_t$ of $\RR^\ell$ such that for all windows $W$ the law
of $Y_t \wedge W$ coincides with the law of the construction in the window.
A global construction for a STIT process was provided in \cite{mnw}.
A STIT tessellation process $Y=(Y_t: t>0)$ is a Markov process 
and each marginal $Y_t$ takes values in $\T$.

\medskip

Here, let us recall  roughly the construction of $Y\land W$ done in \cite{nw}. 

\medskip

Let $\Lambda$ be a (non-zero) measure
on the space of hyperplanes $\hH$ in $\RR^\ell$. It is assumed 
that $\Lambda$ is translation invariant and
possesses the following locally finiteness property:

\begin{equation}
\label{locfin1}
\Lambda([B])<\infty , \; \forall \hbox{ bounded sets } B\subset \RR^\ell, \,
\hbox{ where } \, [B]=\{H\in {\cal H}: H\cap B\neq \emptyset \}.
\end{equation}
(Notation $\subset$ means strict inclusion and $\subseteq$ inclusion
or equality).
It is further
assumed that the support set of $\Lambda$ is such that there is no line in
$\RR^\ell$ with the property that all hyperplanes
of the support are parallel to it (in order
to obtain a.s. bounded cells in the constructed tessellation, cf.
\cite{sw}, Theorem 10.3.2, which can also be applied to STIT tessellations).

\medskip

The assumptions made on $\Lambda$ imply 
$0<\Lambda([W])<\infty \, \hbox{ for every window } \, W$. 
Denote by $\Lambda_{[W]}$ the restriction of $\Lambda$ to $[W]$ and by 
$\Lambda^W=\Lambda ([W])^{-1}\Lambda_{[W]}$ 
the normalized probability measure.
Let us take two independent families of 
independent random variables $D=(d_{n,m}: n,m\in \NN)$ and 
$\tau=(\tau_{n,m}: n,m\in \NN)$, where each $d_{n,m}$ has distribution 
$\Lambda^W$ and each $\tau_{n,m}$ is exponentially distributed with 
parameter $\Lambda ([W])$.

\medskip

\noindent $\bullet$ Even if for $t=0$ the STIT tessellation $Y_0$ is not 
defined in $\RR^\ell$, we define $Y_0\wedge W= \{W\}$ the 
trivial tessellation for the window $W$. Its unique cell is denoted by 
$C^1=W$. 

\medskip

\noindent $\bullet$ Any extant cell has a random lifetime, 
and at the end of its lifetime it is 
divided by a random hyperplane. The lifetime of $W=C^1$ is $\tau_{1,1}$, 
and at that time it is divided by $d_{1,1}$ into two cells denoted by 
$C^2$ and $C^3$. 

\medskip

\noindent $\bullet$ Now, for any cell $C^i$ which is 
generated in the course of the construction, the random sequences 
$(d_{i,m}: m\in \NN)$ and $\tau=(\tau_{i,m}: m\in \NN)$ 
are used, and the following rejection method is applied: 

\medskip

\noindent $\bullet$ When the time $\tau_{i,1}$ is elapsed, 
the random hyperplane $d_{i,1}$ 
is thrown onto the window $W$. If it does not intersect $C^i$ then this 
hyperplane is rejected, and we continue 
until a number $z_i$, which is the first index
$j$ for which a hyperplane $d_{i,j}$ intersects 
and thus divides $C^i$ into two cells $C^{l_1(i)}$, $C^{l_2(i)}$, 
that are called the successors of $C^i$.
Note that this random number $z_i$ is finite a.s. 
Hence the lifetime of $C^i$ is a sum 
$\tau^*(C^i):=\sum_{m=1}^{z_i} \tau_{i,m}$. It is easy to check, see 
\cite{nw}, that $\tau^*(C^i)$ is exponentially 
distributed with parameter $\Lambda ([C^i])$. 
Note that $z_1=1$ and $\tau^*(C^1)=\tau_{1,1}$.

\medskip

\noindent $\bullet$ This procedure is 
performed for any extant cell independently. 
It starts in that moment when the cell is born by division of 
a larger (the predecessor) cell. In order to guarantee independence 
of the division 
processes for the individual cells, the successors of $C^i$ get indexes 
$l_1(i)$, $l_2(i)$ in $\NN$ that are different, and that can be chosen as
the smallest numbers which were not yet used before for other cells. 

\medskip
 
\noindent $\bullet$ For each cell $C^i$ we denote by $\eta=\eta(i)$ the 
number of its ancestors and by \\
$(k_1(i), k_2(i),\ldots ,k_\eta(i))$ the sequence of
indexes of the ancestors of $C^i$. So
$W=C^{k_1(i)} \supset C^{k_2(i)} \supset \ldots \supset C^{k_\eta(i)} \supset C^i$.
Hence, $k_1(i)=1$
The cell $C^i$ is born at time 
$\bt(C^i)=\sum_{l=1}^\eta \tau^*(C^{k_l(i)})$ and it dies at time
$\ut(C^i)=\bt(C^i)+\tau^*(C^{i})$; for $C^1$ this is
$\bt(C^1)=0$ and $\ut(C^1)=\tau^*(C^1)$.
It is useful to put $k_{\eta+1}(i)=i$.

\medskip

With this notation at each time $t>0$ the tessellation $Y_t\land W$ is
constituted by the cells $C^i$ for which $\bt(C^i)\le t$ and $\ut(C^i)> t$. 
It is easy to see that at any time a.s. at most one cell 
dies and so a.s. at most only two cells are born.

\medskip

Now we describe the generated cells as intersections of half-spaces. 
First note that by translation invariance  follows $\Lambda([\{0\}])=0$.
Hence,  all the random hyperplanes 
$(d_{n,m}: n,m\in \NN)$ a.s. do not contain the point $0$. Now, for a 
hyperplane $H\in {\cal H}$ such that $0\not\in H$ 
we denote by $H^+$ and $H^-$ the closed half-spaces 
generated by $H$ with the convention $0\in \Int(H^+)$. 
Hence, $C^{k_{l+1}(i)}= C^{k_l(i)}\cap d_{k_l(i), z_{k_l(i)}}^\pm $, where the 
sign in the upper index determines on which side of the dividing hyperplane the 
cell $C^{k_{l+1}(i)}$ is located. Then any cell can be represented as an 
intersection of $W$ with half-spaces, 
\begin{equation} 
\label{cell} 
C^i = W\, \cap 
\, \bigcap_{l=1}^{\eta(i)} 
\bigcap _{m=1}^{z_{k_l(i)}} d_{k_l(i), m}^{s(k_l(i),m)} \, \cap \, 
\bigcap _{m=1}^{z_i-1} d_{i,m}^{s(i,m)}\,. 
\end{equation} 
In above relation we define the sign $s(j,m) \in \{+,-\}$ by  
the relation $C^j\subset d_{j,m}^{s(j,m)}$, for $j,m\in \NN$.
Notice, that the origin $0\in C^i$  if and only if all signs in the previous 
formula (\ref{cell}) satisfy $s(k_l(i),m)=+$ and $s(i,m)=+$.

\medskip

Obviously the set of cells can be organized as a dyadic tree, by the relation
"$C'$ is a successor of $C$". 
This method of construction is done in \cite{nw}.
For the following it is important to observe that all the 
rejected hyperplanes $d_{k_l(i), m}$ are also included in this intersection 
because the intersection with the appropriate half-spaces does not 
alter the cell. Although the third set 
$\bigcap_{m=1}^{z_{i}-1} d_{i, m}^{s(i,m)}$ in the intersection (\ref{cell}) does 
not modify the resulting set, we also include it 
because we will use this representation later.

\medskip

In \cite{nw} it was shown that there is no explosion, so 
at each time $t>0$ the number of cells
of $Y_t\land W$, denoted by $\xi_t$, is finite a.s.
Renumbering the cells, we write $\{C_t^i: i=1,...,\xi_t\}$ for the set of cells of 
$Y_t\land W$.

\subsection{ Independent increments relation}
\label{indincr}

The name STIT is an abbreviation for "stochastic stability under
the operation of
iteration of tessellations". Closely related to that stability is a certain
independence of increments of the STIT process in time.

\medskip

In order to explain the operation of iteration, for $T\in \T$
we number its cells in the following way.
Assign to each cell a reference point in its interior (e.g.
the Steiner point, see \cite{sw}, p. 613, or
another point that is a.s. uniquely defined). Order the set
of the reference points of all cells of $T$ by its
distance from the origin. For random homogeneous tessellations
this order is a.s. unique. Then number the cells of
$T$ according to this order, starting with
number 1 for the cell which contains the origin.
Thus we write $C(T)^1,C(T)^2,\ \ldots$ for the cells of $T$.

\medskip

For $T\in \T$ and
${\vec R}=(R^m: m\in \NN)\in \T^\NN$,
we define the tessellation $T\boxplus {\vec R}$, referred to
as the iteration of $T$ and ${\vec R}$, by
its set of cells
$$
T\boxplus {\vec R}\!=\!\{ C(T)^k\!\cap \!C({R}^k)^l: \,
k\!=\!1,...;\, l\!=\!1,...;\,
\Int(C(T)^k\!\cap \!C({R}^k)^l)\!\neq \!\emptyset \} .
$$
So, we restrict ${R}^k$ to the cell $C(T)^k$, and this is done for all
$k=1,\ldots $.
The same definition holds when the tessellation and the sequence of
tessellations are restricted to some window.

\medskip

To state the independence relation of the increments of the Markov
process $Y$ of STIT tessellations,
we fix a copy of the random process $Y$ and let
${\vec Y}'=({Y'}^m: m\in \NN)$ be a sequence of
independent copies of $Y$, all of them
being also independent of $Y$. In particular ${Y'}^m\sim Y$. For a fixed
time $s>0$, we set ${\vec Y}'_s=({Y_s'}^m: m\in \NN)$. Then, from the
construction and from the consistency property  of
$Y$ it is straightforward to see that the following property holds
\begin{equation}
\label{iterate}
Y_{t+s} \sim Y_t\boxplus {\vec Y}'_s \ \mbox{ for all }t,s>0\,.
\end{equation}
This relation was firstly stated in Lemma $2$ in \cite{nw}.
It implies $Y_{2t}\sim Y_t\boxplus {\vec Y}'_t$. 
The STIT property means that
\begin{equation}
\label{stit}
Y_{t} \sim 2(Y_t\boxplus {\vec Y}'_t) \ \mbox{ for all }t>0\,,
\end{equation}
so $Y_{t}\sim 2Y_{2t}$.
Here the multiplication with $2$ stands for the transformation
$x\mapsto 2x$, $x\in \RR^\ell$.

\section{ The space of tessellations and the tail $\sigma$-algebra}
\label{Sub1.1}

Let ${\cal C}$ be the set of all compact subsets of 
$\RR^\ell$. We endow $\T$ with the Borel $\sigma$-algebra ${\cal B}(\T)$
of the Fell topology (also known as the topology of closed convergence), 
namely 
$$ 
{\cal B}(\T)=\sigma \left( \{ \{ T\in 
\T :\, \partial T\cap C=\emptyset \} :\, C\in {\cal C} \} \right)\, . 
$$
(As usual, for a class of sets $\I$ we denote by $\sigma({\I})$ the smallest 
$\sigma-$algebra containing $\I$).
Let us fix $\bP$ a probability measure on $(\T, {\cal B}(\T))$.
All the sets are determined  mod$\, \bP$, that is up to a 
$\bP-$negligible set. So, for ${\cal E},\D\in {\cal B}(\T)$ we write ${\cal E}=\D$ mod$\, \bP$ 
if $\bP({\cal E}\Delta \D)=0$. Also for a pair ${\cal B}'$, ${\cal B}''$ of sub-$\sigma$-algebras 
of ${\cal B}(\T)$ we write ${\cal B}'\subseteq {\cal B}''$ mod$\, \bP$ if for all ${\cal E}\in {\cal B}'$
there exists $\D\in {\cal B}''$ such that ${\cal E}=\D$ mod$\, \bP$.

\medskip

For a window $W$ we introduce 
$$ 
{\cal B}(\T_{W})=\sigma \left( \{ \{ T\in \T :\, \partial T\cap C=
\emptyset \} :\, C\subseteq W,\, C\in {\cal C} \} \right) . 
$$ 
By definition ${\cal B}(\T_{W})\subset {\cal B}(\T)$ is a 
sub-$\sigma$-algebra. We
notice that if $W'\subseteq W$ then ${\cal B}(\T_{W'})\subseteq {\cal B}(\T_{W})$.

\medskip

Note that the set $\T\wedge W$ can be endowed with the $\sigma-$field \\
${\cal B}(\T\wedge W)=\sigma \left( \{ \{ T\in \T\wedge W:\, \partial T\cap C=
\emptyset \} :\, C\subseteq W\setminus \partial W,\, C\in {\cal C} \} \right)$.
For all ${\cal E}\in {\cal B}(\T_{W})$ we put ${\cal E}\wedge W=\{T\wedge W: T\in {\cal E}\}$,
which belongs to ${\cal B}(\T\wedge W)$. Denoting the law of the STIT process 
$Y$ by $\PP$ it holds
$$
\forall t>0, \forall \, {\cal E}\in {\cal B}(\T_{W}): \;\;\,
\PP(Y_t\in {\cal E})=\PP(Y_t\wedge W\in {\cal E}\wedge W).
$$
To avoid overburden notation and since there will be no confusion, 
instead of ${\cal E}\wedge W$ in the last formula 
we will only put ${\cal E}$, so it reads
\begin{equation}
\label{simplifnot1}
\forall t>0, \forall \, {\cal E}\in {\cal B}(\T_{W}): \;\;\,
\PP(Y_t\in {\cal E})=\PP(Y_t\wedge W\in {\cal E} )
\end{equation}

Let $(W_n: n\in \NN)$ be an increasing sequence of windows such that 
\begin{equation} 
\label{covandinc} 
\RR^\ell = \bigcup_{n\in \NN} W_n \, \hbox{ and } \, 
\forall \, n\in \NN, \; W_n\subset \Int \, W_{n+1}\,. 
\end{equation} 
We have 
$$ 
{\cal B}(\T_{W_n})\nearrow {\cal B}(\T ) \hbox{ as } n\nearrow \infty\, , 
$$ 
which means $\sigma\left(\bigcup_{n\in \NN} {\cal B}(\T_{W_n})\right)= {\cal B}(\T )$. On 
the other hand, it is easy to check that ${\cal B}(\T)={\cal B}(\T)^{\rm a}$, where 
\begin{equation} 
\label{condaprox} 
{\cal B}(\T)^{\rm a}=\{{\cal E} \in {\cal B}(\T ): \, \forall \epsilon>0, \, 
\exists n\in \NN, \, \exists {\cal E}_n\in {\cal B}(\T_{W_n}) \hbox{ such that } 
\bP({\cal E} \Delta {\cal E}_n)<\epsilon\}\,. 
\end{equation} 

In fact by definition we have ${\cal B}(\T)^{\rm a}\subseteq {\cal B}(\T)$ and
$\bigcup_{n\in \NN} {\cal B}(\T_{W_n})\subseteq {\cal B}(\T)^{\rm a}$. Because
 ${\cal B}(\T)^{\rm a}$ is a $\sigma$-algebra,
then necessarily ${\cal B}(\T)^{\rm a}={\cal B}(\T)$.

\medskip

In order to study the tail $\sigma$-algebra, we will also consider sets of 
tessellations which are determined by their behavior outside a window $W$, 
i.e. in its complement $W^c$. We define the $\sigma$-algebra
$$
{\cal B}(\T_{W^c})=\sigma \left( \{ \{ T\in \T :\, \partial T\cap C=
\emptyset  \} :\, C\subset W^c,\, C\in {\cal C}  \} \right) .
$$
We have ${\cal B}(\T_{W^c})\subset {\cal B}(\T )$. On the other hand, if
$W'\subseteq W$ then ${\cal B}(\T_{{W}^c})\subseteq {\cal B}(\T_{W'^c})$.

\medskip

Let $(W_n: n\in \NN)$ and $(W'_n: n\in \NN)$ be a pair of increasing
sequences of windows satisfying the conditions in (\ref{covandinc}).
Then $\forall n$ $,\, \exists \, m$ such that $W_n\subseteq W'_m$ and $\forall
m$, $\, \exists \,q$ such that $W'_m\subseteq W_q$. This gives
$$
{\cal B}(\T_{W_n^c})\subseteq {\cal B}(\T_{{W'_m}^c})\subseteq {\cal B}(\T_{W_q^c})\,.
$$
Hence
$\cap_{n=1}^\infty {\cal B}(\T_{W_n^c})= \cap_{n=1}^\infty  {\cal B}(\T_{{W'_n}^c})$. 
This equality allows us to define, in analogy with the 
definition done for point processes (see \cite{dvj}, Definition 12.3.IV), 
the tail $\sigma$-algebra on the space of tessellations.
 
\begin{definition}
The {\em tail $\sigma$-algebra} is defined as
${\cal B}_{-\infty}(\T)=\bigcap_{n=1}^\infty  {\cal B}(\T_{W_n^c})$,
where  $(W_n: n\in \NN)$ is an increasing sequence of windows such that for 
all $n\in \NN$, $W_n\subset \Int\, W_{n+1}$, and 
$\RR^\ell = \bigcup_{n\in \NN} W_n$. 
\end{definition}

Note that $(W_n=[-n,n]^\ell :n\in \NN)$ satisfies (\ref{covandinc}), so
it can be used in the above definition and also in the rest of
the paper.

\bigskip

\begin{lemma}
\label{fundtriv1}
Assume that for every window $W'$, all $\D\in{\cal B}(\T_{W'})$ 
and all $\epsilon>0$, there exists a window $\htW$ depending on 
$(W',\D,\epsilon)$, such that
\begin{equation}
\label{condtriv1}
W'\subset \Int\, \htW \; \hbox{ and } \, \forall \, {\cal E}\in{\cal B}(\T_{\htW^c})
\, :\quad \; |\bP(\D\cap {\cal E})-\bP(\D)\bP({\cal E})|<\epsilon.
\end{equation}
Then, the tail $\sigma$-algebra ${\cal B}_{-\infty}(\T)$ is $\bP-$trivial, that is 
\begin{equation}
\label{deftrivi2}
\forall {\cal E}\in {\cal B}_{-\infty}(\T)\,:\;\;\, \bP({\cal E})=0 \hbox{ or } \bP({\cal E})=1\, .
\end{equation}
\end{lemma}

\begin{proof}
Let $\D\in {\cal B}_{-\infty}(\T)$. Let $(W_n: n\in \NN)$ be an increasing sequence 
of windows satisfying (\ref{covandinc}). Let $\epsilon>0$ be fixed. Since  
$\D\in {\cal B}(\T)$, from (\ref{condaprox}) we get,
\begin{equation}
\label{condaprox1}
\exists k\,, \, \exists \D_k\in {\cal B}(\T_{W_k})
\hbox{ such that }\, \bP(\D\Delta \D_k)<\epsilon\,. 
\end{equation}
From hypothesis (\ref{condtriv1}) there exists
a window $\htW$, depending on $(W_k,\D_k,\epsilon)$, such that $W_k\subset \Int\, \htW$
and $\forall {\cal E}\in{\cal B}(\T_{\htW^c})$ it holds 
$|\bP(\D_k\cap {\cal E})-\bP(\D_k)\bP({\cal E})|<\epsilon$.
We know that $\exists \, n\ge k$ such that $\htW\subseteq W_n$. 
So, ${\cal B}(\T_{W_n^c})\subseteq {\cal B}(\T_{\htW^c})$. We then have
$$
\forall {\cal E}\in{\cal B}(\T_{W_n^c})\,:\quad \;
|\bP(\D_k\cap {\cal E})-\bP(\D_k)\bP({\cal E})|<\epsilon\,. 
$$
Since $\D\in {\cal B}_{-\infty}(\T)\subset {\cal B}(\T_{W_n^c})$
we get $|\bP(\D_k\cap \D)-\bP(\D_k)\bP(\D)|<\epsilon$. From (\ref{condaprox1})
we deduce $|\bP(\D\cap \D)-\bP(\D)\bP(\D)|<3\epsilon$. Since this occurs for all
$\epsilon>0$ we conclude $\bP(\D)=\bP(\D)\bP(\D)$, so $\bP(\D)=1$ or $0$. $\Box$
\end{proof}

\medskip

Let us discuss what happens when $\bP$ is translation invariant.
Let $h\in \RR^\ell$. For any set $D\subseteq \RR^\ell$ put 
$D+h=\{x+h: x\in D\}$ and for $T\in \T$ denote by $T+h$ the tessellation
whose boundary is $\partial(T+h)=\partial(T)+h$. For ${\cal E}\subseteq \T$ 
put ${\cal E}+h=\{T+h: T\in {\cal E}\}$. 
For all ${\cal E}\in{\cal B}(\T)$ we have ${\cal E}+h\in {\cal B}(\T)$ because 
$\{C\in {\cal C}\}=\{C+h: C\in {\cal C}\}$. The probability measure
$\bP$ is translation invariant if it satisfies
$\bP({\cal E})=\bP({\cal E}+h)$ for all ${\cal E}\in{\cal B}(\T)$ and $h\in \RR^\ell$. 

\medskip

A set ${\cal E}\in{\cal B}(\T)$ is said to be $(\bP-)$invariant, we put ${\cal E}\in \I(\T)$, 
if $\bP({\cal E} \Delta ({\cal E}+h))=0$ for all $h\in \RR^\ell$. Note that if
${\cal E}\in \I(\T)$, $\D\in {\cal B}(\T)$ and ${\cal E}=\D$ mod$\, \bP$, then $\D\in \I(\T)$.
It is easily checked that $\I(\T) \subseteq {\cal B}(\T)$ is a 
sub-$\sigma$-algebra, the {\em invariant $\sigma$-algebra} 
(see e.g. \cite{dvj}). We have the inclusion relation,
 \begin{equation} 
\label{incinv1} 
\I(\T)\subseteq {\cal B}_{-\infty}(\T)\, \hbox{mod} \bP\,. 
\end{equation} 
We note that (\ref{incinv1}) corresponds to propositions
in \cite{dvj} (pp. 206--210) for random measures.
For completeness we will prove (\ref{incinv1}). 
We denote ${\bf 1}=(1,...,1) \in \RR^\ell$,
so  $a{\bf 1}=(a,...,a)$ for $a\in \RR$.
Fix the sequence $(W_n=[-n,n]^\ell :n\in \NN)$. 
Note that if $m>2n$ then $W_n-m{\bf 1}\subset W_n^c$. 

\medskip

Let ${\cal E}\in \I(\T)$: For all $n\in \NN$, there exists $k=k(n)> n$ and 
${\cal E}_k\in {\cal B}(\T_{W_k})$ such that $\bP({\cal E}\Delta {\cal E}_k)<2^{-n}$. 
Since ${\cal E}_k\in \sigma(\{\{ T\in \T :\, \partial T\cap C=
\emptyset  \}: C\subseteq W_k, C\in {\cal C}\}$ we have that for 
$N>2k$, 
\begin{eqnarray*} 
&{}&  
{\cal E}_k\!-\!N {\bf 1}\  \in \  \sigma(\{\{ T\in \T :\, \partial T\cap C=
\emptyset  \}: C\subseteq W_k\!-\!N {\bf 1}, 
C\in {\cal C}\})\\
&{ \subseteq }& \sigma(\{\{ T\in \T :\, \partial T\cap C=
\emptyset  \}: C\subset W_{k}^c, C\in {\cal C}\}) \ = \
{\cal B}(\T_{W_{k}^c})\,. 
\end{eqnarray*}

Since $\bP({\cal E}\Delta ({\cal E}\!-\!N {\bf 1}))=0$ and 
$\bP(({\cal E}\!-\!N{\bf 1})\Delta ({\cal E}_k\!-\!N {\bf 1}))=\bP({\cal E}\Delta {\cal E}_k)< 2^{-n}$
 we get, 
$$ 
\forall N\!>\!2k:\;\; \bP({\cal E} \Delta ({\cal E}_k\!-\!N{\bf 1}))\le 
\bP({\cal E}\Delta ({\cal E}\!-\!N{\bf 1}))+\bP(({\cal E}\!-\!N{\bf 1})
\Delta ({\cal E}_k\!-\!N{\bf 1}))<2^{-n}. 
$$ 
We have $k>1$. Take $N=k^2$ and use $k^2>2k$ to get 
$$
\bP({\cal E}\Delta({\cal E}_k\!-\!k^2 {\bf 1}))\le 2^{-n}
\hbox{ and } {\cal E}_k\!-\!k^2{\bf 1}\in {\cal B}(\T_{W_{k}^c}). 
$$
Define $\D_m=\bigcup_{n> m}({\cal E}_{k(n)}\!-\!k(n)^2{\bf 1})$ for $m\ge 1$. 
This sequence of sets satisfies $\bP({\cal E}\Delta \D_m)\le 2^{-m}$, 
$\D_m\in {\cal B}(\T_{W^c_{k(m)}})$ and $\D_m$ decreases with $m$\,. 
We conclude that $\D=\bigcap_{m>1} \D_m$ satisfies $\bP({\cal E}\Delta \D)=0$ 
and $\D\in {\cal B}_{-\infty}(\T)$. Hence, relation (\ref{incinv1}) holds.

\medskip

So, when the tail $\sigma$-algebra ${\cal B}_{-\infty}(\T)$ is $\bP-$trivial, then
$\bP$ is ergodic with respect to translations, 
because every invariant set ${\cal E}\in \I(\T)$ also belongs to ${\cal B}_{-\infty}(\T)$
and so $\bP({\cal E})=0$ or $\bP({\cal E})=1$. We also have that if $\bP$ is 
translation invariant and ${\cal B}_{-\infty}(\T)$ is $\bP-$trivial, then
the action of translations is mixing. That is,
\begin{equation}
\label{prop11}
{\cal B}_{-\infty}(\T) \hbox{ is } \bP-\hbox{trivial }\Rightarrow \, 
\forall \, \D,  {\cal E}\in {\cal B}(\T)\,:\;
\lim\limits_{|h|\to \infty}\bP(\D\cap ({\cal E}+h))=\bP(\D)\bP({\cal E}).
\end{equation}
This result is shown for random measures in Proposition 12.3.V. in \cite{dvj}. But for completeness
let us prove it. Let $\D, \, {\cal E}\in {\cal B}(\T)$ and fix $\epsilon >0$. From 
(\ref{condaprox}) follows the existence of $k$ and ${\cal E}_k\in {\cal B}(\T_{W_k})$ such 
that $\bP({\cal E} \Delta {\cal E}_k)<\epsilon$. For every $h\in \RR^\ell$ we have 
$$
|\bP(\D\cap ({\cal E}+h)-\bP(\D\cap ({\cal E}_k+h)|<\epsilon .
$$ 
Our choice $W_n=[-n,n]^\ell$ implies that for all 
$h\in \RR^\ell$, $N\in \NN$ with $N>k$ and 
$|h|> (N+k) \sqrt{\ell}$ we have 
${\cal E}_k+ h \in {\cal B}(\T_{W_{N}^c})$, $W_k\cap W_N^c 
=\emptyset$, and so 
\begin{equation} 
\label{formcondN} \bP(\D\cap ({\cal E}_k+h))= 
\EE(\EE({\bf 1}_{{\cal E}_k+h}{\bf 1}_{\D}\, | \,{\cal B}(\T_{W_{N}^c}))= 
\EE({\bf 1}_{{\cal E}_k+h}\EE({\bf 1}_{\D} \, | \, {\cal B}(\T_{W_{N}^c})). 
\end{equation} 
The Decreasing Martingale Theorem (see e.g. \cite{par}) yields 
$$ 
\lim\limits_{N\to \infty} 
\EE({\bf 1}_{\D} | {\cal B}(\T_{W_{N}^c})= \EE({\bf 1}_{\D} \, | \, 
{\cal B}_{-\infty}(\T)) \hbox{ in } L^1(\bP)\,. 
$$ 
Since ${\cal B}_{-\infty}(\T)$ is assumed 
to be $\bP-$trivial we have $\EE({\bf 1}_{\D} \, | \, {\cal B}_{-\infty}(\T))=\bP(\D)$. 
So $\lim\limits_{N\to \infty} \EE({\bf 1}_{\D} | {\cal B}(\T_{W_{N}^c})=\bP(\D)$ in 
$L^1(\bP)$. Let $N$ be sufficiently large in order that $||\EE({\bf 1}_{\D} | 
{\cal B}(\T_{W_{N}^c})-\bP(\D)||_{1} <\epsilon$. Then for $|h|> (N+k) \sqrt{\ell}$ 
we can use (\ref{formcondN}) to obtain, 
$$
|\bP(\D\cap ({\cal E}_k+h))-\bP(\D)\bP({\cal E}_k+h)|\le 
||\EE({\bf 1}_{\D} | {\cal B}(\T_{W_{N}^c})-\bP(\D)||_{1}<\epsilon\,.
$$
Since $\bP({\cal E}_k+h)=\bP({\cal E}_k)$ we deduce
$\bP(\D\cap ({\cal E}_k+h))\to \bP(\D)\bP({\cal E}_k)$ as $h\to \infty$.
We conclude $\bP(\D\cap ({\cal E}+h))\to \bP(\D)\bP({\cal E})$ 
as $h\to \infty$, so mixing is shown.

\section{Main results}
\label{trivitail}

As already defined, $Y=(Y_t:\, t>0)$ denotes a STIT tessellation
process, defined by the measure $\Lambda$ on $\hH$ in $\RR^\ell$ 
which satisfies the properties described in Section \ref{constr}.

\begin{theorem}
\label{alphacoeff} 
Let $W'$ be a window. Then  
\begin{eqnarray*} 
&{}& \forall t>0,\, \forall \varepsilon >0,\,  
\exists \hbox{ a window } \htW \hbox{ such that } 
W'\subset \Int \htW,\, \hbox{ and  }\\ 
&{}& 
\forall \D\in {\cal B}(\T_{W'}) ,\ \forall {\cal E}\in {\cal B}(\T_{\, \htW^c})\,:\;\;
|\PP (Y_t \!\in \! \D \!\cap \! {\cal E}) - \PP (Y_t \!\in \! \D) 
\PP (Y_t \!\in \! {\cal E})| < \varepsilon \,. 
\end{eqnarray*}
\end{theorem}

\medskip

Denote by ${\bP}^t$ the marginal law of
$Y_t$ in $\T$, that is 
$$
\forall \, {\cal E}\in {\cal B}(\T)\,:\;\;\, {\bP}^t({\cal E})=\PP(Y_t\in {\cal E})\,.
$$
We say that at time $t>0$ the tail $\sigma$-algebra ${\cal B}_{-\infty}(\T)$ is trivial
for the STIT process $Y$ if  ${\bP}^t$ is trivial.

\medskip

Theorem \ref{alphacoeff} implies that ${\bP}^t$ 
satisfies the sufficient condition 
(\ref{condtriv1}), which by Lemma \ref{fundtriv1} implies 
the triviality of the tail $\sigma$-algebra ${\cal B}_{-\infty}(\T)$. 
Hence, the following result holds.

\begin{theorem}
\label{trivtail}
For all $t>0$ the tail $\sigma$-algebra is trivial for the STIT process $Y$,
that is $\forall \, {\cal E}\in {\cal B}_{-\infty}(\T)$ we have 
$\PP(Y_t\in {\cal E})=0$ or $\PP(Y_t\in {\cal E})=1$.
\end{theorem}

Relation (\ref{prop11}) ensures that Theorem \ref{trivtail} is stronger than
the mixing property shown by Lachi{\`e}ze-Rey \cite{lr}.

\section{Proof of Theorem \ref{alphacoeff}} 
\label{proofs}

Let $W'$ be a window. Since the measure $\Lambda$ is supposed to 
be translation invariant, without loss
of generality it can be assumed that the
origin $0\in \Int(W')$.

\medskip

Let $W$ be a window such that $W'\subset \Int(W)$. A key idea 
is the investigation of the probability that the 
window $W'$ and the complement $W^c$ are 
separated by the STIT process $Y\land W$ (we say $W'$ encapsulated within $W$) 
and thus the construction inside $W'$ 
and outside $W$, respectively, are approximately independent.

\medskip

For simplicity, denote by $C_t=C^1_t$ the (a.s. uniquely determined)
cell of tessellation $Y_t\land W$ that
contains the origin in its interior. Obviously, $C_0 =W$. 
Note that $C_t$ decreases as time $t$ increases. On the other
hand since $W'\subset W$, when we consider the STIT on $W'$
we can take $(Y_t\land W)\land W'$. 

\begin{definition}
\label{encaps}
Let be $W',\ W$ be two windows with $0\in W'\subset \Int(W)$ and let $t>0$. 
We say that {\em $W'$ is encapsulated inside $W$ at time $t$} 
if the cell $C_t$ that contains $0$ in $Y_t\land W$ is such that,
\begin{equation}
\label{def_encaps}
W' \subseteq C_t\subset \Int(W)\,.
\end{equation}
We write $W' |_t W$ if $W'$ is encapsulated inside $W$ at time $t$. 
\end{definition}

We denote the {\em encapsulation time} by
$$
{\aS}(W',W) =\inf \{ t>0 :\, W'|_t W   \}\,,
$$
where as usual we put ${\aS}(W',W)=\infty$ when $\{t>0 :\, W'|_t W\}=\emptyset$. 
Encapsulation of $W'$ inside $W$ means that ${\aS}(W',W)<\infty$
or equivalently $W' |_t W$ for some $t>0$. That is, where 
the boundaries $\partial W'$ and $\partial W$ 
are completely separated by facets of the 0-cell 
before the smaller window $W'$ is hit for the first time 
by a facet of the STIT tessellation.

\medskip

We have $\{{\aS}(W',W)\le t\}\in \sigma(Y_s:s\le t)$, so it is a stopping time.
Hence the variable ${\aS}(W',W)$ is also a stopping time for the processes 
$Y\wedge W$. 
On the other hand notice that the distribution of ${\aS}(W',W)$ 
does neither depend on the method of construction of the STIT process $Y$, nor 
on the window $\widetilde W$ where the construction is performed as long as 
$W\subseteq {\widetilde W}$. In several proofs of the results we will assume 
that the starting process is $Y$, but in some others ones 
we will start from the STIT process $Y\wedge W$, 
as it occurs in Lemma \ref{lemma_encapsprob1}.

\medskip

Note that even if in the STIT construction we have "independence after separation", 
it has to be considered that the tessellation outside $W$ also depends on 
the process until the separation time.

\medskip

For two Borel sets $A,B\subset \RR ^\ell$ we denote by 
$$
[A|B]=\{H\!\in \!{\hH}: \left(A\!\subset \!\Int(H^+) 
\!\land \!B \!\subset \!\Int(H^-)\right) 
\lor \left((A\!\subset \!\Int(H^-) \!\land \!B\!\subset \!\Int(H^+)\right)\} ,
$$ 
the set of all hyperplanes that separate $A$ and $B$. This set is 
a Borel set in $\hH$.

\medskip

For a window $W$ (which is defined to be a convex polytope) we denote by 
$\{f^W_a: \, a=1,\ldots,q\}$ the $(\ell-1)$-dimensional facets of $W$. Let $W'$ 
be another window such that $W'\subset \Int(W)$. We denote by $$ G_a=[W'|f^W_a], 
\, a=1,\ldots,q\,; $$ the set of hyperplanes that separate $W'$ from the facet 
$f^W_a$ of $W$. Note that all these sets are non-empty, and they are not 
necessarily pairwise disjoint.

\medskip

There exists a finite family $\{G'_a: a=1,\ldots, q\}$  of pairwise disjoint 
nonempty measurable sets that satisfy
\begin{equation}
\label{partition}
\forall a\in \{ 1,\ldots ,q\} ,\, 
G'_a\subseteq G_a. 
\end{equation}

(E.g. we can choose $G'_a=G_a \setminus \bigcup_{b< a}G_b$, 
or we can partition $\bigcup_{a=1}^q G_a$ alternatively.)

\medskip

\begin{lemma}
\label{lemma_encapsprob1}
Let $W'$, $W$ be two compact convex polytopes, with $W'\subset \Int(W)$. 
Let $\{G'_a: a=1,\ldots, q\}$ be a finite class of nonempty disjoint measurable sets 
satisfying (\ref{partition}) and such that
$\Lambda (G'_a)>0$ for all $a= 1,\ldots ,q$. Then,
\begin{eqnarray}
\nonumber
\PP ({\aS}(W' , W)\leq t )\geq &{}&
{\rm e}^{-t\Lambda ([W'])} \, \prod_{a=1}^{q} 
\left(1-{\rm e}^{-t\Lambda (G'_a)} \right) \\
\label{encapsprob} 
&+&
\int_0^t \Lambda ([W']) {\rm e}^{-x \Lambda ([W'])}  
\prod_{a=1}^{q} \left(1- {\rm e}^{-x \Lambda (G'_a)} \right){\rm d} x \,. 
\end{eqnarray}
\end{lemma}

\begin{proof}
Our starting point in the proof is the construction of
the process $Y\wedge W$.
Because we assume $0\in W'$ we focus on the genesis of the 
0-cell $C_t=C_t^1$, $t\geq 0$, only. We use the representation (\ref{cell}), 
and we emphasize that in this intersection also participate
the positive half-spaces of those hyperplanes $d_{i, m}$ 
which are rejected in the construction.

\medskip

Now, we consider a Poisson point process 
$$
\Phi = \{ (d_{ m}, S_{ m}):\,  m \in \NN\}
$$ 
on $[W] \times [0,\infty )$ 
with the intensity measure $\Lambda^W\otimes  \Lambda ({[W]}) \, \lambda_+$, 
where $\lambda_+$ denotes the Lebesgue measure 
on $\RR_+ =[0,\infty )$.
This point process can be considered as a marked hyperplane process where the marks 
are birth-times (or as a "rain of hyperplanes"). 
This choice of the intensity measure corresponds to the families 
$D$ and $\tau$ of random variables in Subsection \ref{constr}: 
the interval between two sequential births of hyperplanes is 
exponentially distributed with parameter $\Lambda ({[W]})$, 
and the law of the born hyperplanes is $\Lambda^W$. 
Thus the $S_{m}$ are sums of i.i.d. exponentially 
distributed random variables which are independent of the $d_{ m}$. 
This corresponds to one of the standard methods to construct  (marked) 
Poisson point processes  (cf. e.g. \cite{king}).
Note that this Poisson process is used for the construction of the 0-cell exclusively.

\medskip

Let $\eta$ denote the number of ancestors of $C_t$ with indexes 
${k_1},\ldots ,{k_\eta}$, and $Z_{k_l}=\sum_{m=1}^l z_{k_l}$, 
$l=1,\ldots , \eta +1$. Thus we 
can write formula (\ref{cell}) for the 0-cell as
\begin{equation} 
\label{zerocell} 
C_t = W\,  \cap 
\, \bigcap_{l=1}^{\eta} \ \,  
\bigcap _{m=Z_{k_{l-1}}+1}^{Z_{k_l}} d_{ m}^+ \, \cap \, 
\bigcap _{m=Z_{k_{\eta}}+1}^{Z_{k_{\eta +1}}-1} d_{m}^+\,
= W\, \cap \, \bigcap_{m:S_m <t} d_m^+ \, .
\end{equation}

Define the random times 
\begin{eqnarray*}
\sigma '&=&\min \{ S: \exists (d,S)\in \Phi :d\in [W'] \}  \, \hbox{ and }\\
\sigma_a&=&\min \{ S: \exists (d,S)\in \Phi :d\in G'_a \}, \,a=1,\ldots q .
\end{eqnarray*}
These are the first times that a hyperplane of $\Phi$ falls into 
the respective sets. Note that these minima exist 
and are greater than $0$ because all $\Lambda_{[W]}(G'_a)>0$, and we are working on 
a bounded window $W$ and $\Lambda$ is assumed to be locally finite, 
so $\Lambda^W$ is a probability measure. Let
$$
\M=\max \{ \sigma_a:\, a=1,\ldots q\} \,.
$$
By definition for all $a=1,\ldots , q$ there exists a
$(d_{(a)}, S_{(a)})\in \Phi$ with $d_{(a)}\in G'_a$ and $S_{(a)}\leq \M$.
Then, $f^W_{a}\subset d_{(a)}^-$ and $C_{S_{(a)}}\subseteq d_{(a)}^+$.
Since $C_{\M}\subseteq C_{S_{(a)}}$ we deduce
\begin{equation}
\label{suffcell1}
C_{\M} \subseteq  \bigcap_{a=1}^{q}  d_{(a)}^+ \subset \Int(W)\,.
\end{equation}
On the other hand, if $\sigma' \geq \M $ then $W'$ is not intersected
until the time $\M=\max \{\sigma_a:\, a=1,\ldots , q\}$, so we have
$W'\subseteq C_{\M}$. Then
$$
W'\subseteq C_{\M} \subset \Int(W)\,.
$$
We have shown
$$
\left\{ \M \leq \sigma' \right\} \,
\subseteq \, \left\{ {\aS}(W' , W)\leq \M  \right\}\,.
$$
This relation implies straightforwardly the inclusion,
\begin{equation}
\label{inclevent1}
 \left\{ \M \leq \min \{\sigma', t \} \right\} \,
\subseteq \, \left\{ {\aS}(W' , W)\leq t  \right\}\,.
\end{equation}
Indeed, from $\M \leq \sigma'$ we get
${\aS}(W' , W)\leq \M$, and we use $\M\le t$ to get relation (\ref{inclevent1}).
We deduce
\begin{equation}
\label{bound1}
\PP ({\aS}(W' , W)\leq t)\geq \PP (\M \leq \min \{ \sigma' ,t \} )\,.
\end{equation}

Now, the sets $[W']$, $G'_a$, $a=1,\ldots , q$, are pairwise disjoint 
and therefore the restricted Poisson point processes 
$\Phi \cap ([W'] \times [0,\infty ) )$,  
$\Phi \cap (G'_a\times [0,\infty ) )$, $a=1,\ldots , q$, 
are independent. Hence $\sigma'$, $\sigma_a,\, a=1,\ldots ,q$,
are independent random variables. Then $\sigma'$ and $\M$ are independent
and we get
\begin{eqnarray*}
\PP (\M \leq \min \{ \sigma' , t \} )
&=& \PP (\M \leq \sigma ' \leq t ) + \PP (\M \leq t \leq \sigma ' )\\
&=& \PP (\M \leq \sigma ' \leq t ) + \PP (\M \leq t)
\PP( t \leq \sigma')\,.
\end{eqnarray*}
Since $\sigma'$, $\sigma_a,\, a=1,\ldots ,q$, 
are exponentially distributed 
with the respective parameters $\Lambda ([W'])$, $\Lambda (G'_a)>0$, 
$a=1,\ldots , q$, we find
$$
\PP( t \leq \sigma ' )= {\rm e}^{-t\Lambda ([W'])} \, \hbox{ and } \,
\PP (\M \leq t)= \prod_{a=1}^{q} \left(1- {\rm e}^{-t\Lambda (G'_a)} \right) . 
$$
Now, by denoting the density functions of $\sigma'$ by $p'$ and 
those ones of $\sigma_{i}$ by $p_{i}$ respectively, we get
\begin{eqnarray*}
\PP (\M \leq \sigma ' \leq t ) 
&= & \int_0^t p'(x) \left( \int_0^x p_{1}(x_{1}) {\rm d} x_{1} \ldots 
   \int_0^x p_{q}(x_{q}) {\rm d} x_{q}     \right) {\rm d} x \\    
&= & \int_0^t \Lambda ([W'])  {\rm e}^{-x \Lambda ([W'])}  
\prod_{a=1}^{q} \left(1-{\rm e}^{-x \Lambda (G'_a)} \right){\rm d} x \,.
\end{eqnarray*}
Therefore, formula (\ref{encapsprob}) follows. $\Box$
\end{proof}

\medskip

\begin{remark}
$(i)$ We emphasize that $\M \leq \min \{ \sigma ', t \}$ is 
sufficient but not necessary for ${\aS}(W' , W)\leq t $. There can 
also be other ways to encapsulate $W'$ within $W$ than separating 
the complete facets of $W$ by single hyperplanes. Alternative geometric 
constructions are possible. 

\medskip

\noindent $(ii)$ It is well known in convex geometry, 
that $[W'|f^W_a]\not= \emptyset$ for all $a=1,\ldots ,q$. 
But depending on the support of the measure $\Lambda$ (in particular, if 
$\Lambda$ is concentrated on a set of hyperplanes with only finitely many 
directions) there can be windows $W$ such that $\Lambda ([W'|f^W_a])=0$ 
for some $a$.
In such cases the bound given in Lemma 
\ref{lemma_encapsprob1} is useless. Therefore, in the following, $W$ will be 
adapted to $\Lambda$ in order to have all $\Lambda (G'_a)>0$. But here we will 
not try to find an optimal $W$ in the sense that the quantities 
$\Lambda (G'_a)$ could be somehow maximized.

\medskip

\noindent $(iii)$ As an example consider the particular measure $\Lambda_\bot = 
\sum_{c=1}^\ell g_c \delta_c$, with $g_c>0$, $\delta_c$ the translation 
invariant measure on the set of all hyperplanes that are orthogonal to the $c$-th 
coordinate axis in $\RR^\ell$, with the normalization $\delta_c ([s_c])=1$, 
where $s_c$ is a linear segment of length $1$ and 
parallel to the $c$-th coordinate axis. 
Let  $W'=[-\alpha,\alpha]^\ell$, $W=[-\beta,\beta]^\ell$ be two windows
with $0<\alpha<\beta$. 
Then we can choose the sets $G'_a=G_a$ 
if the facet $f^W_a$ of $W$ is orthogonal to the 
$c$-th coordinate axis, $a=1, \ldots , 2 \ell$. 
We have $\Lambda_\bot (G'_a)= g_c(\beta-\alpha)$. Simple geometric 
considerations yield,  
\begin{equation}
\label{asequevent} 
\left\{\M \leq \min \{ \sigma ',t \} \right\} \, = \, \left\{{\aS}(W', W)\leq t 
\right\} \quad a.s. 
\end{equation} 
and hence for $\Lambda_\bot$ we have the equality sign in 
(\ref{encapsprob}). 
\end{remark}

\medskip

We will use the following parameterization of hyperplanes. Let
${\SaS}^{\ell -1}$ be the unit hypersphere in $\RR^\ell$. 
For $H\in {\cal H}$, $d(h)\in \RR$ denotes its signed distance 
from the origin and $u(h)\in {\SaS}^{\ell -1}$ is its normal direction. 
We denote by $H(u,d)$ the hyperplane 
with the respective parameters $(u,d)\in \RR\times {\SaS}^{\ell -1}$. 
The image of the non-zero, locally finite and
translation invariant measure $\Lambda$ with respect to 
this parameterization can be written as the product measure
\begin{equation}
\label{prodmeas}
\gamma \cdot \lambda \otimes \theta, 
\end{equation}
where $\gamma >0$ is a constant, $\lambda$ is the Lebesgue measure on 
$\RR$ and $\theta$ is an even  probability measure on 
${\SaS}^{\ell -1}$ (cf, e.g. \cite{sw}, Theorem 4.4.1 and Theorem 13.2.12). 
Here $\theta$ is even means 
$\theta (A)= \theta (-A)$ for all Borel sets $A\subseteq {\SaS}^{\ell -1}$.
The property that there is no line in $\RR^\ell$ 
such that all hyperplanes of the support of $\Lambda$ 
are parallel to it, is equivalent to the property that $\theta$ is 
not concentrated on a great subsphere of 
${\SaS}^{\ell -1}$, i.e there is no one-dimensional 
subspace $L_1$ of $\RR^\ell$ (with the orthogonal complement 
$L_1^\bot$) such that the support of $\theta$ equals 
${\mathbb G}= L_1^\bot \cap {\SaS}^{\ell -1}$.

\medskip

Recall that $W'$ is a window with $0\in \Int(W')$.

\begin{lemma}
\label{goodW}
There exists a compact convex 
polytope $W$ with facets $f^W_1,\ldots ,f^W_{2\ell}$ and 
$W'\subset \Int(W)$, and pairwise disjoint sets $G'_a\subseteq [W'|f^W_a]$ 
such that $\Lambda (G'_a)>0$ for all $a=1,\ldots , 2\ell$.
\end{lemma}

\begin{proof}
For $u\in {\SaS}^{\ell -1}$ we denote by $H_{W'}(u)$ the supporting 
(i.e. tangential) hyperplane to $W'$
with normal direction $u\in {\SaS}^{\ell -1}$. By $h_{W'}(u)$
we denote the distance from the origin to $H_{W'}(u)$. 
This is the support function of $W'$, 
$h_{W'}(u)=\max\{<x,u>: x\in W'\}$
(see \cite{sw}, p. 600). With this notation we have 
$H_{W'}(u)=H(u,h_{W'}(u))$. Note that for $d\in \RR$ the hyperplane
$H(u,d)$ is parallel to $H_{W'}(u)$ at signed distance $d$ from 
the origin. 

\medskip

The shape of the window $W$ will depend on $\Lambda$. 
We use some ideas of the proof of Theorem 10.3.2 in \cite{sw}. 
Under the given assumptions on the support of $\Lambda$ there exist points 
$u_1,\ldots ,u_{2\ell}\in {\SaS}^{\ell -1}$ 
which all belong to the support of $\theta$ and 
$0\in \Int(conv\{ u_1,\ldots ,u_{2\ell}\})$, 
i.e. the origin is in the interior of the 
convex hull. Now, the facets $f^W_a$ of $W$ are chosen to have normals $u_a$, 
and their distance from the origin is $h_{W'}(u_a)+3$, $a=1,\ldots , 2\ell$. 
Formally,
$$
W=\bigcap_{a=1}^{2\ell }  H(u_a,h_{W'}(u_a)+3)^+ .
$$ 
Notice that the described condition on the choice of the 
directions $u_1,\ldots ,u_{2\ell}$ guarantees that $W$ is bounded 
(see the proof of Theorem 10.3.2 in \cite{sw}).

\smallskip

The definition of the support of a measure and some continuity arguments 
(applied to sets of hyperplanes) 
yield that for all $u_a$ there are pairwise disjoint neighborhoods 
$U_a \subset {\SaS}^{\ell -1} $ such that 
$\theta (U_a)>0$, and for the sets of hyperplanes 
$$
G'_a = \{ H \in {\cal H}:\, u(H)\in U_a, \,
h_{W'}(u_a)+1 < d(H)< h_{W'}(u_a)+2\}\,,
$$ 
it holds $G'_a \subset [W'|f^W_a]$. Hence $\Lambda ([W'|f^W_a])
\geq \Lambda (G'_a) = \gamma \, \theta (U_a) >0$ for all $a=1,\ldots , 2\ell$. 
Since the $U_a$ are pairwise disjoint, 
also the sets $G'_a, a=1,\ldots , 2\ell ,$ have this property. $\Box$
\end{proof}

\medskip

\begin{remark}
Because the directional distribution $\theta$ is assumed to be an even measure, 
in the construction above one can choose $u_{\ell+a}=-u_a$, $a=1,\ldots ,\ell$. 
Then the facets $f^W_a$ and $f^W_{\ell +a}$ are parallel. 
\end{remark}

\medskip

\begin{lemma}
\label{goodWeps}
The compact convex polytope $W$ constructed in Lemma \ref{goodW} also 
satisfies the following property: $\forall \,  \varepsilon >0$, 
$\exists \, t^*(\varepsilon) >0$ such that the following
encapsulation time relation holds, 
\begin{equation}
\label{goodWepsi}
\forall \,  s\in (0,t^*(\varepsilon)], \; \exists \, r(s)\ge 1
\hbox{ such that } \forall r\ge r(s):\;\; 
\PP ({\aS}(W',r W)\leq s )> 1-\varepsilon\,.
\end{equation}
\end{lemma}

\begin{proof}
Let us use the notation introduced in the Lemma \ref{goodW} and in its proof. 
For $r>0$ we set $r G'_a=\{ rh:\, h\in G'_a\}$. 
Then, elementary linear algebra yields that 
$G'_a \subset [W'|f^W_a]$ implies 
$r G'_a \subset [W'| r f^W_a]$ for all $r>1$. 
Furthermore, from (\ref{prodmeas}) we find 
$$
\forall \, a=1,\ldots , 2\ell\,: \;\;\, 
\Lambda ([W'|\, r \,f^W_a])\geq \Lambda (r G'_a) 
= \gamma \, r \, \theta (U_a) \,.
$$ 
Now denote by 
$$
L =\min \{ \Lambda (G'_a) = \gamma \theta (U_a) : \ a=1,\ldots , 2\ell \} .
$$ 
We have $L>0$, and (\ref{encapsprob}) yields
\begin{eqnarray*}
&{}& \PP ({\aS}(W' , r W)\leq s ) \\ 
&{}& \geq  {\rm e}^{-s \Lambda ([W'])} \,  
\left(  1- {\rm e}^{- s r L} \right) ^{2\ell } +
\int_0^s \Lambda ([W'])  {\rm e}^{-x \Lambda ([W'])}   
\left(  1- {\rm e}^{-x r L} \right) ^{2\ell} {\rm d} x \\ 
&{}& >  {\rm e}^{-s \Lambda ([W'])} \,  
\left(  1- {\rm e}^{- s r L} \right) ^{2\ell } . 
\end{eqnarray*}

Note that
\begin{equation}
\label{defint1}
\forall \varepsilon \in (0,1), \, \exists \, t^*(\varepsilon)>0
\hbox{ such that } {\rm e}^{- t^*(\varepsilon) \Lambda ([W'])}
> \sqrt{1-\varepsilon}.
\end{equation}
Then for all $s\in (0,t^*(\varepsilon)]$ we have
${\rm e}^{-s \Lambda ([W'])}> \sqrt{1-\varepsilon}$. 
Furthermore, for any such $s\in (0,t^*(\varepsilon)]$ there is an 
$r(s)\ge 1$ with $\left( 1- {\rm e}^{- s r(s) L} \right) ^{2\ell } > 
\sqrt{1-\varepsilon}$. This finishes the proof. $\Box$
\end{proof}

\medskip

\begin{lemma}
\label{nojump}
For all $W'$, all $t>0$ and all $\varepsilon >0$ and  $t^*(\varepsilon)$  such that (\ref{goodWepsi}) holds, there exists 
$t_1=t_1(\varepsilon,t)\in (0,\min\{t,t^*(\varepsilon)\})$  
 such that for all $t_2\in (0,t_1]$,
$$
\PP (Y\wedge W' \mbox{ has no jump in } [t-t_2,t)) > 1-\varepsilon \,.
$$
\end{lemma}

\begin{proof} 
Let $\{C_t^{'i}: i=1,...,\xi'_t\}$ be the family of cells of 
the pure jump Markov process $Y_t\wedge W'$. The lifetimes of 
$C_t^{'i}$ are exponentially distributed with the parameters 
$\Lambda ([C_t^{'i}])$ and they are conditionally independent 
conditioned that a certain set of cells is given at time $t$.
Then, given a certain set of cells at $t$, 
the minimum of the lifetimes is exponentially distributed 
with parameter $\zeta_t = \sum_{i=1}^{\xi'_t} \Lambda ([C_t^{'i}])$.

\medskip

Notice that $\zeta_t$ is monotonically increasing in 
$t$, because if at some time a cell $C'$ is divided into the
cells $C'',\, C'''$  we have $C'=C''\cup C'''$ and 
$[C']=[C'']\cup [C''']$. Then, 
by subadditivity of $\Lambda$,
$$
\Lambda ([C'])=\Lambda ([C'']\cup [C'''])\le 
\Lambda ([C''])+\Lambda ([C'''])\,.
$$
Since the process $Y\wedge W'$ has no
explosion, for any fixed $t>0$ $\, \exists \, x_0>0$ 
such that for all $s\in [0,t]$ we have 
$\PP(\zeta_s \leq x_0) > \sqrt{1-\varepsilon}$. We fix 
$t_1=t_1(\varepsilon,t)\in (0,\min\{t,t^*(\varepsilon)\})$ as a value 
which also satisfies 
${\rm e}^{-t_1(\varepsilon) \, x_0 } > \sqrt{1-\varepsilon}$. 
This yields for all $t_2\in (0,t_1]$ that 
\begin{eqnarray*} 
&{}&\PP(Y\wedge W' \mbox{ has no jump in } [t-t_2,t))\\ 
&{}&\ge 
\quad \PP(Y\wedge W' \mbox{ has no jump in } [t-t_1,t))\\
&{}&= 
\int_0^\infty \PP(Y\wedge W' \mbox{ has no jump in } [t-t_1,t)
\, | \, \zeta_{t-t_1} =x) \PP(\zeta_{t-t_1}\in dx)\\
&{}&\ge
\int_0^{x_0} \PP(Y\wedge W' \mbox{ has no jump in } [t-t_1,t)
\, | \, \zeta_{t-t_1} =x) \PP(\zeta_{t-t_1}\in dx)\\
&{}&= 
\int_0^{x_0} {\rm e}^{-t_1 \, x} \PP(\zeta_{t-t_1}\in dx)
\geq {\rm e}^{-t_1 \, x_0 } \PP (\zeta_{t-t_1} \leq x_0 ) 
> 1-\varepsilon \,.
\end{eqnarray*} $\Box$
\end{proof}

\medskip

In the sequel for $t>0$ and $\varepsilon >0$ the quantity
$t_1=t_1(\varepsilon,t)$ {is the one given by Lemma \ref{nojump}.

\medskip

Recall that under the identification ${\cal E}$ with ${\cal E}\wedge W$ 
we can write $\PP(Y\in {\cal E})=\PP(Y\wedge W'\in {\cal E})$ for all
${\cal E}\in {\cal B}(\T_{W'})$, see (\ref{simplifnot1}). This identification
also allows us to put for all ${\cal E}\in {\cal B}(\T_{W'}), \D\in {\cal B}(\T_{W^c}), s>0$:
$$
\PP(Y_t\in {\cal E}\cap \D, {\aS}(W',W)<s)=\PP(Y_t\wedge W'\in {\cal E}, Y_t\in \D, 
{\aS}(W',W)<s)\, .
$$

\medskip

\begin{lemma}
\label{eventsmallwindow}
For all $t>0$, $\varepsilon >0$ and $t_2\in (0,t_1]$, we have 
$$
\forall \, s\in (0,t_2],\,  \forall \, {\cal E}\in {\cal B}(\T_{W'})\,:\;\;\,
|\PP (Y_t \in {\cal E})-\PP (Y_{t-s}\in {\cal E})| < \varepsilon .
$$
\end{lemma}

\begin{proof}
We have 
$$
\forall s\in (0,t_2]\,:\;\;
\{Y\wedge W' \hbox{ has no jump in } [t-t_2,t)\}\subseteq
\{Y\wedge W' \hbox{ has no jump in }[t-s,t) \}\,.
$$ 
Then,
\begin{eqnarray*}
&{}& \;\;
\{Y_t\wedge W' \in {\cal E}, Y\wedge W' \mbox{ has no jump in } [t-t_2,t)\}\\
&{}& \;\;
=\{Y_{t-s}\wedge W' \in {\cal E}, Y\wedge W' \mbox{ has no jump in } [t-t_2,t)\}.
\end{eqnarray*}
Therefore $\{Y_t \wedge W' \in {\cal E}\} \Delta
\{Y_{t-s}\wedge W' \in {\cal E}\}\subseteq 
\{Y\wedge W' \mbox{ has some jump in } [t-t_2,t)\}$.
By using the relation 
$|\PP(\Gamma)-\PP(\Theta)|\le \PP(\Gamma\Delta \Theta)$, the result follows. $\Box$
\end{proof}

\medskip

In the following results $W$ is a window such that $W'\subset \Int(W)$.
We use the notation ${\aS}={\aS}(W',W)$ for the encapsulation time.

\begin{proposition}
\label{corevents}
For all $t>0$, $\varepsilon >0$ and $t_2\in (0,t_1]$, we have
$$
\forall {\cal E}\in {\cal B}(\T_{W'}):\;\;
|\PP (Y_t \in {\cal E})-\PP (Y_t \in {\cal E} \, | \, {\aS}<t_2 )| 
< \varepsilon\,.
$$
\end{proposition}

\begin{proof}
Let us first show,
\begin{equation}
\label{condA}
\forall s\in (0,t),\, 
\forall {\cal E}\in {\cal B}(\T_{W'}): \;\; 
\PP(Y_t \in {\cal E} \, | \; {\aS}= s) =\PP(Y_{t-s}\in {\cal E})\,.
\end{equation}
Let ${\vec Y}'=({Y'}^m: m\in \NN)$ be a sequence of
independent copies of $Y$, and
also independent of $Y$. Relation (\ref{iterate}) yields 
$$
Y_{t}\wedge W' \ \sim \  (Y_s\boxplus {\vec Y}'_{t-s})\wedge W'\  
\sim \ (Y_s\wedge W') \boxplus {\vec Y}'_{t-s}\,. 
$$
On the event ${\aS}=s$ we have $W'\subseteq C^1_s$, this last is the cell 
containing the origin at time $s$, and thus $Y_s\wedge W'=\{ W'\}$.
Hence, on ${\aS}=s$ we have 
$(Y_s\wedge W') \boxplus {\vec Y}'_{t-s}\sim Y'^{1}_{t-s} \wedge W'$. 
Therefore,
\begin{equation}
\label{simplifynot2}
\PP(Y_t \wedge W'\in {\cal E} \, | \,  {\aS}= s)=
\PP(Y'^{1}_{t-s} \wedge W' \in {\cal E} \, | \, {\aS}= s)
=\PP(Y'^{1}_{t-s} \wedge W'\in {\cal E}).
\end{equation}
Since $Y'^{1}_{t-s}\sim Y_{t-s}$, relation (\ref{condA}) is satisfied.
Hence,
\begin{eqnarray*}
&{}&\quad  
|\PP (Y_t\in {\cal E})-\PP (Y_{t}\in {\cal E} \, | \, {\aS}<t_2 )| \\
&{}&= \left| \PP (Y_t \in {\cal E})- \frac{1}{\PP ({\aS}<t_2)} 
\int_0^{t_2} \PP (Y_{t}\in {\cal E} \, | \, S=s)\; \PP({\aS}\in ds) \right| \\
&{}&= \left| \PP (Y_t \in {\cal E})- \frac{1}{\PP ({\aS}<t_2)} 
\int_0^{t_2} \PP (Y_{t-s}\in {\cal E})\; \PP({\aS}\in ds) \right| \\
&{}&= \frac{1}{\PP (S<t_2)}
\left| \int_0^{t_2} \PP (Y_t \in {\cal E})- 
\PP (Y_{t-s}\in {\cal E} )\; \PP({\aS}\in ds) \right| \\
&{}&\leq  \frac{1}{\PP ({\aS}<t_2)} \int_0^{t_2} |\PP (Y_t \in {\cal E})-   
\PP (Y_{t-s}\in {\cal E})|\; \PP(\aS\in ds) < \varepsilon \,,               
\end{eqnarray*} 
where in the last inequality we use Lemma \ref{eventsmallwindow}. $\Box$
\end{proof}

\medskip

Lemma \ref{eventsmallwindow}, Proposition \ref{corevents} and 
relation (\ref{condA}) obviously imply the following result.

\begin{corollary}
\label{condAS}
For all $t>0$, $\varepsilon >0$, $t_2\in (0,t_1]$ and all $s\in (0,t_2)$ it holds
\begin{eqnarray*}
\forall {\cal E}\in {\cal B}(T_{W'}):\;\, &{}& |\PP (Y_{t-s} \in {\cal E})-
\PP (Y_{t}\in {\cal E} \, | \,{\aS}<t_2 )|\\
&{}& \; =|\PP (Y_{t} \in {\cal E} \, | \, {\aS}=s)-
\PP (Y_{t}\in {\cal E} \, | \, {\aS}<t_2)|<2 \varepsilon.
\end{eqnarray*}
\end{corollary}

\medskip

\begin{lemma}
\label{integralbound}
For all $t>0$, $\varepsilon >0$ and $t_2\in (0,t_1]$ we have
for all $\D\in {\cal B}(\T_{W'}),\, {\cal E}\in {\cal B}(\T_{W^c})$,
$$
|\PP(Y_t \in \D  \cap {\cal E} \, | \, {\aS}< t_2)
- \PP(Y_t \in \D \, | \, {\aS}< t_2)  
\PP( Y_t \in {\cal E} \, | \, {\aS}< t_2) | < 2 \varepsilon.
$$
\end{lemma}

\begin{proof}
Firstly, we show the following conditional independence property,
\begin{eqnarray}
\nonumber
&{}& \forall \D\in {\cal B}(\T_{W'}),\, {\cal E}\in {\cal B}(\T_{W^c}):\\
\label{condind1}
&{}& 
\PP(Y_t\in \D \cap {\cal E} \, | {\aS}\!=\! s)= \PP(Y_t \in \D \, |  {\aS}\!=\! s) 
\PP( Y_t\in {\cal E} \, |  {\aS}\!=\! s) .
\end{eqnarray}
We use the notation introduced in the proof of Proposition \ref{corevents} and we also shortly write ${\cal E}$ instead of ${\cal E} \wedge W^c$.
Also the arguments are close to the ones in 
the proof of relation (\ref{condA}).

\medskip

Recall that on the event ${\aS}=s$ we have $W'\subseteq C^1_s$.
By $Y_s\boxplus (Y'^m_{t-s}: m \geq 2)$ we mean 
that the tessellations $Y'^m_{t-s}$ are nested only 
into the cells $C^m_s$ of $Y_s$ with $m \geq 2$, and not 
into the cell $C^1_s$. From (\ref{condA}), (\ref{simplifynot2}) and 
the independence of the random variables 
$Y_s$, $Y'^m_{t-s}$, $m \geq 1$ we obtain,
\begin{eqnarray*}
&{}&\quad \PP(Y_t\! \wedge \! W'\!\in \! \D , Y_t \in  {\cal E} \, | \, {\aS}= s)\\
&{}&= \PP((Y_s\boxplus {\vec Y}'_{t-s})\! \wedge \! W'\!\in \! \D, 
(Y_s\boxplus {\vec Y}'_{t-s})\in {\cal E} \, | \,  {\aS}= s)\\
&{}&= \PP(Y'^1_{t-s}\! \wedge \! W'\!\in \! \D, (Y_s\boxplus 
(Y'^m_{t-s}: m \geq 2))\in {\cal E} \, | \,  {\aS}= s)\\
&{}&= \PP(Y'^1_{t-s}\! \wedge \! W'\!\in \! \D ) 
\PP (Y_s\boxplus (Y'^m_{t-s}: m \geq 2))\in {\cal E} \, | \, {\aS}=s)\\
&{}&= \PP(Y_t\! \wedge \! W'\!\in \! \D \, | \, {\aS}\!=\! s) 
\PP( Y_t\in {\cal E} \, | \,  {\aS}= s)\,.
\end{eqnarray*} 
Then (\ref{condind1}) is verified. By using this equality and 
Corollary \ref{condAS} we find,
\begin{eqnarray*}
&{}& \quad \PP(Y_t \in \D \, | \,  {\aS}< t_2)  
\PP( Y_t \in {\cal E} \, | \,  {\aS}< t_2)  - 2 \varepsilon \\
&{}& \leq (\PP(Y_t \in \D \, | \,  {\aS}< t_2) - 2 \varepsilon)  
\PP( Y_t \in {\cal E} \, | \,  {\aS}< t_2)  \\
&{}& = \frac{1}{\PP (S<t_2)} \int_0^{t_2}(\PP(Y_t \in \D \, | \,  
{\aS}< t_2) - 2 \varepsilon) 
\PP( Y_t \in {\cal E} \, | \,  {\aS}=s) \PP({\aS}\in ds) \\
&{}& < \frac{1}{\PP ({\aS}<t_2)} \int_0^{t_2} 
\PP(Y_t \in \D \, | \,  {\aS}=s) 
\PP( Y_t \in {\cal E} \, | \,  {\aS}=s) \PP({\aS}\in ds) \\
&{}&= \frac{1}{\PP ({\aS}<t_2)} \int_0^{t_2} 
\PP(Y_t \in \D \cap {\cal E} \, | \, {\aS}=s)\, \PP({\aS}\in ds)\\
&{}&=\PP(Y_t \in \D \cap {\cal E} \, | \, {\aS}< t_2).
\end{eqnarray*}
In an analogous way it is proven,
$$
\PP(Y_t \in \D\cap  {\cal E} \, | {\aS}\!<\! t_2) 
<\PP(Y_t \in \D \, |  {\aS}< t_2) 
\PP( Y_t \in {\cal E} \, |  {\aS}< t_2)  + 2 \varepsilon, 
$$
which finishes the proof. $\Box$
\end{proof}

\medskip

We summarize. The window $W'$ was fixed and there was no loss of
generality in assuming $0\in \Int(W')$. Let $t>0$ and $\varepsilon>0$ be fixed.
We construct $t_1=t_1(\varepsilon,t)\in (0,\min\{t, t^*(\varepsilon)\})$. 
Now let $W$ be as in Lemma \ref{goodW}. 
Note that for all $t_2\in (0,t_1(\varepsilon)]$ we have 
${\rm e}^{-t_2 \Lambda ([W'])}> \sqrt{1-\varepsilon}$. Then, from 
Lemma \ref{goodWeps} we get that for all $t_2\in (0,t_1(\varepsilon)]$ 
there exists $r(t_2)\ge 1$ such that
$\PP ({\aS}(W',r\htW)\!<\! t_2)>1-\varepsilon$ for all $r\ge r(t_2)$.

\medskip

Now we fix $t_2\in (0,t_1(\varepsilon)]$ and define $\htW=r(t_2) W$.
We have $W'\subset \Int(\htW)$. From Lemma \ref{integralbound}
we deduce the following result.

\begin{corollary}
\label{corjunto} 
For all $\D\in {\cal B}(\T_{W'}),\, {\cal E}\in {\cal B}(\T_{\htW^c})$ it is satisfied
$$
|\PP(Y_t \in \D\cap {\cal E} \, | \,  {\aS}\!< t_2)
- \PP(Y_t \in \D \, | \, {\aS}\!< t_2)
\PP( Y_t \in {\cal E} \, | \, {\aS}<t_2) |< 2 \varepsilon\,.
$$
\end{corollary}

\medskip

\begin{proposition}
\label{alphamix}
For the window $W'$, for all $t>0$ and $\varepsilon>0$, the window $\htW$ satisfies, 
$$
\forall\, \D\in {\cal B}(\T_{W'}),\, {\cal E}\in {\cal B}(\T_{\htW^c})\,:
|\PP (Y_t \in \D\cap {\cal E}) - 
\PP (Y_t \in \D) \PP (Y_t \in {\cal E})| < 4\varepsilon \,.
$$
\end{proposition}

\begin{proof}
Denote
$$
\Gamma= \{ Y_t \in  \D \}, \;\,
\Theta= \{ Y_t \in {\cal E} \},
\;\, \Upsilon= \{ {\aS}(W',\htW)< t_2 \}.
$$
We have
\begin{equation}
\label{relCDG}
|\PP (\Gamma \cap \Theta \, | \, \Upsilon) - \PP (\Gamma \, | \, \Upsilon)
\PP (\Theta \, | \, \Upsilon) | < 2 \varepsilon \hbox{ and }
\PP(\Upsilon)>1-\varepsilon\,.
\end{equation}
Observe that $\PP(\Upsilon)>1-\varepsilon$ implies,
$\PP (\Xi)-\PP (\Xi \cap \Upsilon)<\varepsilon$ and
$\PP (\Xi)-\PP (\Xi \cap \Upsilon)\PP(\Upsilon)<2\varepsilon$ for all
events $\Xi$, in particular when $\Xi$ is $\Gamma$, $\Theta$ or
$\Gamma\cap \Theta$.

\medskip

The first relation in (\ref{relCDG}) obviously implies 
$$\bigg|\PP (\Gamma \cap \Theta \, | \, \Upsilon) - 
\PP (\Gamma \, | \, \Upsilon)
\PP (\Theta \, | \, \Upsilon) \bigg| \  \PP (\Upsilon)^2 < 2 \varepsilon.$$
 Then,
\begin{eqnarray*}
&{}&\quad \PP (\Gamma\! \cap \! \Theta) \!-\! 4 \varepsilon \\
&{}&< \PP (\Gamma \! \cap \! \Theta)\!-\! 
\PP (\Gamma \! \cap \! \Theta \, | \, \Upsilon) \PP (\Upsilon)^2 
\!+\! \PP (\Gamma \, | \, \Upsilon) 
\PP (\Theta \, | \, \Upsilon) \PP (\Upsilon)^2 \!-\! 2 \varepsilon \\
&{}&= \PP (\Gamma \! \cap \! \Theta)\!-
\! \PP (\Gamma \! \cap \! \Theta\! \cap \! \Upsilon) 
\PP (\Upsilon) \!+\! 
\PP (\Gamma\! \cap \! \Upsilon) \PP (\Theta\! \cap \! \Upsilon)  \!-\! 
\PP (\Gamma) \PP (\Theta) \!+\! 
\PP (\Gamma) \PP (\Theta) \!-\! 2 \varepsilon \\ 
&{}&< \PP (\Gamma) \PP (\Theta)\,.
\end{eqnarray*}
In an analogous way it is shown that
$\PP (\Gamma)\PP (\Theta)< \PP (\Gamma\cap \Theta) + 4 \varepsilon$. Hence, 
the result is proven.     $\Box$         
\end{proof}

\medskip

Proposition \ref{alphamix} yields the proof of Theorem \ref{alphacoeff}, by 
substituting $4\, \varepsilon$ by $\varepsilon$.

\section{Comparison of STIT and Poisson hyperplane tessellations (PHT)}
Intuitively, we expect a gradual difference in the mixing properties 
of STIT and PHT. Hyperplanes are unbounded, while the maximum 
faces (also referred to as I-faces) in a STIT tessellation are always bounded. 
But these I-faces tend to be very large (it was already shown 
in \cite{mnwIseg} 
for the planar case that the length of the typical I-segment has a finite 
expectation but an infinite second moment). But since we have shown that the 
tail $\sigma$-algebra for STIT is trivial, then STIT has a short range 
dependence in the sense of \cite{dvj}.

\medskip

One aspect is the following. For PHT, Schneider and 
Weil \cite{sw} (Section 10.5) showed, that it is mixing if the directional 
distribution ${\theta}$ (see (\ref{prodmeas})) has zero mass on all 
great subspheres of ${\SaS}_+^{\ell -1}$. And they contribute an 
example of a tessellation where 
this last condition is not fulfilled 
and which is not mixing. In contrast to this, 
Lachi\`eze-Rey proved in \cite{lr} that STIT is mixing, 
for all ${\theta}$ which are not concentrated on a great subsphere.

\medskip

Concerning the tail $\sigma$-algebra, we have 
Theorem \ref{trivtail} for the STIT tessellations. 
In contrast, for the PHT the tail $\sigma$-algebra is not trivial.

\begin{lemma}
\label{poisstail}
Let $Y^{\rm PHT}$ denote a Poisson hyperplane tessellation with 
intensity measure $\Lambda$ which is a non-zero, 
locally finite and translation invariant measure on $\hH$, and  
it is  assumed that the support set of $\Lambda$ is such that there is 
no line in $\RR^\ell$ with the property that all hyperplanes of the support 
are parallel to it. Then the tail $\sigma$-algebra 
is not trivial with respect to the distribution of $Y^{\rm PHT}$.
\end{lemma} 
\begin{proof}
Let $(W_n: n\in \NN)$ be an increasing sequence of windows such that 
for all $n\in \NN$,
$W_n\subset \Int\, W_{n+1}$, and $\RR^\ell = \bigcup_{n\in \NN} W_n$.
Let $B_1$ be the unit ball in $\RR^\ell$ centered at $0$. Consider the 
following event in ${\cal B}(\T )$:
$$
\D:=\{\exists \, \hbox{ an hyperplane in } Y^{\rm PHT} \hbox{ which 
intersects } B_1\}.
$$
Because outside of any bounded window $W_n$ all the 
hyperplanes belonging to $Y^{\rm PHT}$ can be identified and thus it 
can be decided (in a measurable way) whether there is a hyperplane which 
intersects $B_1$, we have that  $\D\in {\cal B}(\T_{W_n^c})$ for all $n\in \NN$, 
and hence $\D\in {\cal B}_{-\infty}(\T)$.
We have $\PP(\D)=1-{\rm e}^{-\Lambda ([B_1])}$,
and this probability is neither 0 nor 1. $\Box$
\end{proof}

\vspace{1cm}

{\bf Acknowledgments}. The authors thank
Lothar Heinrich for helpful hints and discussion.
The authors are indebted for the support of Program Basal CMM 
from CONICYT (Chile) and by DAAD (Germany).

\end{document}